\documentclass{article}
\usepackage[utf8]{inputenc}
\usepackage{enumerate}
\usepackage{amssymb, bm}
\usepackage{amsthm}
\usepackage{amsmath}
\usepackage{tikz-cd} 
\usepackage{tikz}
\usepackage{bbm}
 \usepackage[all]{xy}
\newtheorem{mythm}{Theorem}
\newtheorem{mylem}{Lemma}

\newtheorem{myex}{Example}
\newtheorem{mycor}{Corollary}
\newtheorem{mydef}{Definition}
\newtheorem{myrem}{Remark}
\title{Note on asymptotic behaviour of the canonical ring}
\author{Xiaojun WU}
\date{\today}

\begin{document}
\maketitle
  \def\tors{\mathrm{Tors}}
\def\cI{\mathcal{I}}
\def\Z{\mathbb{Z}}
\def\Q{\mathbb{Q}}  \def\C{\mathbb{C}}
 \def\R{\mathbb{R}}
 \def\N{\mathbb{N}}
 \def\H{\mathbb{H}}
  \def\P{\mathbb{P}}
 \def\rC{\mathcal{C}}
  \def\nd{\mathrm{nd}}
  \def\d{\partial}
 \def\dbar{{\overline{\partial}}}
\def\dzbar{{\overline{dz}}}
 \def\ii{\mathrm{i}}
  \def\d{\partial}
 \def\dbar{{\overline{\partial}}}
\def\dzbar{{\overline{dz}}}
\def \ddbar {\partial \overline{\partial}}
\def\cN{\mathcal{N}}
\def\cM{\mathcal{M}}
\def\cE{\mathcal{E}}  \def\cO{\mathcal{O}}
\def\cF{\mathcal{F}}
\def\cS{\mathcal{S}}
\def\cQ{\mathcal{Q}}
\def\P{\mathbb{P}}
\def\cI{\mathcal{I}}
\def \loc{\mathrm{loc}}
\def \cC{\mathcal{C}}
\bibliographystyle{plain}
\def \dim{\mathrm{dim}}
\def \Sing{\mathrm{Sing}}
\def \Id{\mathrm{Id}}
\def \rank{\mathrm{rank}}
\def \tr{\mathrm{tr}}
\def \Ric{\mathrm{Ric}}
\def \Vol{\mathrm{Vol}}
\def \RHS{\mathrm{RHS}}
\def \liminf{\mathrm{liminf}}
\def \ker{\mathrm{Ker}}
\def \nn{\mathrm{nn}}
\def \div{\mathrm{div}}
\def \PSSF{\mathrm{PSSF}}
\def\stateparagraph{\vskip7pt plus 2pt minus 1pt\noindent}
\begin{abstract}
The theory of the Oukounkov body is a useful tool for studying the asymptotic behaviour of the canonical ring of a line bundle over a projective manifold.
In this note, combined with the algebraic reduction, we study the asymptotic behaviour of the canonical ring of a line bundle over an arbitrary compact normal irreducible complex space.
\end{abstract}

The general setting is the following.
Let $X$ be a compact normal irreducible (reduced) complex space.
We denote the meromorphic function field over $ X$ by $\cM(X)$ over $X$.
By \cite{Thi54}, \cite{Rem56}, \cite{AS71}, it is a finitely generated extension over $\C$.
In particular, there exists a (reduced irreducible) projective variety $Y$ such that $\cM(X)$ is isomorphic to the rational function field of $Y$ (called a model of $\cM(X)$).
Any two models are bimeromorphic.
Let $L$ be a line bundle over $X$ (or a Cartier divisor if the space is singular). 
Classically, the canonical ring of $L$ is defined
to be 
$$R(X,L):= \oplus_{k \geq 0} H^0(X, kL)$$
and we denote $\N(L):=\{k \in \N, h^0(X,kL) \neq 0\}$.
In all the note, we assume that $L$ is $\Q-$effective, that is $\N(L) \neq \{0\}$.
Otherwise, the Kodaira-Iitaka dimension of $L$ is $-\infty$.

Let $v$ be a valuation of $\cM(X)$
(i.e. a group homomorphism $v: \cM(X)^* \to \Z^k$ for some $k \in \N^*$ such that
$v(f + g) \geq \min \{v(f ), v(g)\}$ for any $f, g \in \cM(X)^*$ with respect to some total order on $\Z^k$).
Assume in this paragraph that $X$ is projective.
The theory of Okounkov's body produces a tool to study the asymptotic behaviour of the canonical ring of $L$ in analysing the image of the valuation.
The theory is developed independently by Lazarsfeld and
Mustaţǎ \cite{LM09} and Kaveh and Khovanskii \cite{KK12} which gives a systematic study of the Okounkov's construction \cite{Oko96},  \cite{Oko00}.
In particular, we can show that the limit
$$\lim_{k \in \N(L), k \to \infty} \frac{h^0(X,kL)}{k^{\kappa(L)}}$$
exists where $\kappa(L)$ is the Kodaira-Iitaka dimension of $L$.
Notice that by the definition of the Kodaira-Iitaka dimension of $L$, we have a priori that the limit superior 
$$\limsup_{k \in \N(L), k \to \infty} \frac{h^0(X,kL)}{k^{\kappa(L)}}$$
exists and is strictly positive.
Notice that the canonical ring is a bimeromorphic invariant by the projection formula. 
In other words, if $\nu: \tilde{X} \to X$ is a modification of $X$, $R(X,L) \simeq R(\tilde{X}, \nu^* L)$.


In our general situation, the centre of a valuation does not necessarily exist on $X$.
Notice that the existence of the centre in the projective setting is deduced from the valuation characterisation of the properness of a scheme.
Therefore, the centre exists on any model of $\cM(X)$ since it is projective, although it is not necessarily a bimeromorphic model for a non-projective irreducible complex space.
So to study the asymptotic behaviour of the canonical ring $R(X, L)$ by the valuation approach,
we choose a model such that the canonical ring of $L$ is isomorphic to some canonical ring of some $\Q-$line bundle over this model.
The existence of such $\Q-$line bundle is shown by the following fundamental Theorem 1 of Campana.

To prepare the proof, let's recall some definitions introduced by Campana.
\begin{mydef}(Definition 1.21 \cite{Cam04})

Let $f: X \to Y$ be a holomorphic fibration between compact manifolds (i.e. surjective with connected fibres), and
$S$ be an effective divisor on $X$. We say that $S$ is partially supported on the fibres of $f$ if $f(S) \neq Y$ and if for any irreducible component $T$ of $f (S)$
of codimension one in $Y$, then $f^{-1}(T)$
contains an irreducible component mapping onto $T$ by $f$, but not contained in support of $S$.
\end{mydef}

We have the following basic property.
\begin{mylem}(Lemma 1.22 \cite{Cam04})

Let $f : X \to Y$ be a holomorphic fibration between
manifolds, and $S$ be a divisor of $X$ partially supported on the fibres of $f$. Let
$L$ be a line bundle on $Y$. The natural injection of sheaves $L \subset f_* ( f^* (L) \otimes \cO( S))$
is an isomorphism.

\end{mylem}
\begin{proof}
We sketch the proof for the reader's convenience.
Let $U$ be a local Stein open set on $Y$ such that $L$ is trivial over $U$.  
Since $S$ is partially supported on the fibres of $f$ (more precisely, $f(S) \neq Y$)
and $U$ is Stein,
there exists an effective divisor $T$ on $U$ such that $\cO(S) \subset f^* \cO(T)$.
Thus
$$f_* \cO(S) \subset \cO(T) \subset \cM_U.$$
Any section of $f_* \cO(S)$ on $U$ is some meromorphic function on $U$ whose pull back by $f$ is a meromorphic function on $f^{-1}(U)$ with at worst poles along $S$. 
Since $S$ is partially supported on the fibres of $f$, the meromorphic function has to be holomorphic (i.e. $f_* \cO(S)=\cO_U$).
\end{proof}

\begin{mydef}(Definition 1.2 \cite{Cam04})

Assume that $f : X \to Y$ is a holomorphic fibration between (connected) compact manifolds.
An irreducible divisor $D$ on $X$ is said to be $f$-exceptional if $f (D)$ has
codimension at least 2 in $Y$. We say that $f : X \to Y$ is neat if there moreover exists a bimeromorphic
holomorphic map $u : X \to X'$ with $X'$ smooth such that each $f-$exceptional
divisor of $X$ is also $u-$exceptional. 
\end{mydef}
Note that an $f-$exceptional divisor is partially supported on the fibres of $f$.
We have the following lemma by resolution of singularities \cite{Hir64} and Hironaka's flattening theorem \cite{Hir75}.
\begin{mylem}(Lemma 1.3 \cite{Cam04})
Assume that $f : X \to Y$ is a holomorphic fibration between (connected) compact manifolds.
Then, there exists a neat model
 $f' : X' \to  Y'$ of $f$ and a bimeromorphic map $u :
 X' \to X$ 
 $$\begin{tikzcd}
X' \arrow[r, "u"] \arrow[d, "f'"] & X \arrow[d, "f"] \\
Y' \arrow[r]           & Y          
\end{tikzcd}$$
such that each $f'-$exceptional
divisor of $X'$ is also $u-$exceptional.
\end{mylem}
The basic property of neat morphism is the following.
\begin{mylem}
Assume that $f : X \to Y$ is a neat holomorphic fibration between (connected) compact manifolds.
Let $u : X \to X'$ with $X'$ smooth such that each $f-$exceptional
divisor of $X$ is also $u-$exceptional. 
Let $E$ be a $f-$exceptional
divisor of $X$ (hence $u-$exceptional) and $L$ a line bundle over $X'$.
Then the restriction induces an isomorphism
$$H^0(X, u^* L) \simeq H^0(X \setminus E, u^*L ).$$
In particular, the multiplication with the canonical section $s_E$ of $E$ induces an isomorphism
$$H^0(X, u^* L) \simeq H^0(X, u^*L+E).$$
\end{mylem}
\begin{proof}
It is enough to show the surjectivity. 
Since $u: X \to X'$ is a bimeromorphic morphism, there exists a closed analytic subset $S \subset X'$ of codimension at least 2 such that the restriction of $u$ on $u^{-1}(X' \setminus S)$ is biholomorphic.

The first statement follows from the following diagram
$$
\begin{tikzcd}
{H^0(X',L)} \arrow[r, "u^*"] \arrow[d, "\simeq"] & {H^0(X,u^*L)} \arrow[r] & {H^0(X\setminus E,u^*L)} \arrow[d]    \\
{H^0(X' \setminus  S,L)} \arrow[rr, "\simeq"]  &                         & {H^0(X \setminus f^{-1}( S), u^*L)}
\end{tikzcd}
$$
since the right arrow is injective.
Note that the left arrow is an isomorphism by Hartogs' theorem. 
For any $s \in H^0(X, u^*L+E)$, by the first statement, $(s/s_E)|_{X \setminus E}=s'|_{X\setminus E}$ for some $s' \in H^0(X, u^*L)$ which implies $s=s's_E$ and thus the second isomorphism.
\end{proof}

We also recall the definition of non-polar divisors defined in \cite{FF79}.
For more information on the non-polar divisor, we refer to the paper \cite{Cam82}.
\begin{mydef}
An irreducible divisor on a complex manifold $X$ is called non-polar if it is not contained in the pole on any meromorphic function on $X$.
\end{mydef}
For any compact connected manifold $X$, we have the following algebraic reduction (cf. page 25 of \cite{Ueno}, or  Lemma 1 (page 163) of \cite{Cam81})
$$\begin{tikzcd}
X' \arrow[d,"a"]\arrow[r, "m"] & X \\ A
\end{tikzcd}$$
with $X'$ a smooth bimeromorphic model of $X$, $m$ a proper modification, $A$ a smooth projective variety and $a$ 
some surjective holomorphic map with connected fibres such that 
$\cM(X) \cong \cM(A) \cong \cM(X')$.
In this case, an irreducible divisor on $X'$ is non-polar if and only if its image under $a$ is $A$.
\begin{mylem}
Under the notations above, let $N$ be a non-polar divisor over $X'$ and $S$ be a divisor over $X'$ without a common irreducible component with $N$.
Then for any $p \geq 0$, the multiplication with the canonical section $s_{pN}$ of $pN$ induces an isomorphism
$$H^0(X',S)\simeq H^0(X',S+pN).$$
\end{mylem}
\begin{proof}
The sections of $H^0(X', S+pN)$ are the meromorphic functions $f \in \cM(A)$ such that
$$a^* \div(f)+S+pN \geq 0$$
which is equivalent to the condition that $a^* \div(f)+S \geq 0$ since $N$ is non polar and $S$ has no common irreducible component with $N$.
Thus
$$H^0(X',S)\simeq H^0(X',S+pN)$$
for any $p \geq 0$.
\end{proof}

\begin{mythm}
Let $L$ be a line bundle over a compact irreducible normal complex space $X$.
Let $k \in \N^*$ such that 
$kL$ is effective.
There exists a smooth projective variety $A$ (independent of $k$) an algebraic reduction of $X$ such that the rational function field of $A$ is isomorphic to $\cM(X)$ and there exists a $\Q-$effective divisor $D$ over $A$ such that for $m>0$ sufficient divisible,
$$H^0(A, mD) \cong H^0(X, mkL).$$
Such $A$ is unique up to bimeromorphism.
\end{mythm} 
\begin{proof}

Up to a possible desingularisation of $X$, we can assume $X$ to be a compact connected complex manifold.
Take the notations before Lemma 4.
By Lemma 1, we can assume that $a$ is neat.

We claim that there exist $\Q-$effective divisors (as $kL$ is effective) such that
$$k m^* L+R= a^* D+E$$
where $R$ is an effective, $a$-exceptional divisor, $E$ is a sum of non-polar divisors $N$ and an effective divisor $\PSSF(a)$ partially supported on the fibres of $a$.
Note that $R$ is thus $m-$exceptional since $a$ is neat.
Thus for $l$ sufficient divisible
$$H^0(X, lkl) =H^0(X', m^*(lkL))=H^0(X', m^*(lkL)+lR)$$
by Lemma 3.

The construction of the above decomposition is as follows.
Let $D'$ be an irreducible component of $m^*(kL)$ such that $G:=a(D')$ is an irreducible divisor of $A$.
Then
$$a^* G=\sum_i k_i D_i+R$$
with $R$ $a-$exceptional and $D_i$ irreducible divisors such that $a(D_i)=G$.
Let 
$$G' =\sum_i g_i D_i$$
be the maximal effective divisor as a linear combination of $D_i$ such that $m^*(kL)-G'$ is effective.
Note that $G'$ is not trivial since $G' \geq D'$.

Similarly, let $N$ be the maximal effective divisor as a linear combination of non-polar irreducible components of $m^*(kL)$ such that $m^*(kL)-N$ is effective.
Here the support of $N$
is contained in support of $m^*(kL)$; thus, the set of such divisors is finite.
In general, the set of non-polar divisors is always finite by the results of \cite{Cam82}.

Define $m_G:=\min_i \frac{g_i}{k_i}$.
Then $m_G=0$ if and only if there exists $i$ such that $g_i=0$.
In this case, by definition, $G'$ is partially supported on the fibres of $a$.
Define $R_G=0$ in this case.
If $m_G>0$, there exists an $a-$exceptional $\Q-$effective divisor $R_G$ such that
$G'- m_G a^* G+R_G$ is $\Q-$effective and is partially supported on the fibres of $a$.

Consider
$m^*(kL)- \sum_G m_G a^* G$ where the sum is taken over all irreducible divisors $G$ that can be written as $a(D')$ for some irreducible component $D'$ of $m^*(kL)$.
Define
$$D:=\sum_G m_G  G.$$
Then $m^*(kL)- \sum_G m_G a^* G+R$ with $R:= \sum_G R_G$ (which is $a-$exceptional, $\Q-$effective) is a sum of non polar divisor $N$ and a $\Q-$effective divisor
$$\PSSF(a):=m^*(kL)+R-a^* D-N$$
partially supported on the fibres of $a$.

To relate the canonical ring of $kL$ to the canonical ring of some line bundle on $A$,
for any $p>0$ sufficient divisible such that $p m_G \in \Z$ for any $m_G$,
consider
$$H^0(X, pkL)=H^0(X', p m^* kL+pR)=H^0(X', pa^*D+pN+p\PSSF(a))=H^0(A, p D).$$
The third equality follows from Lemma 1 and 4.
(In fact, it is enough to take $p$ as a common multiple of all $k_i$. Note that up to $\Q-$linear equivalence, $\frac{1}{k} D$ is uniquely determined by $L$.)
\end{proof}
\begin{myrem}{\rm
In general, we hope to use the above theorem to construct the Okounkov body over an arbitrary compact normal irreducible complex space.
Let $L$ be a line bundle over a compact normal irreducible complex space $X$.
Let $v$ be a valuation of $\cM(X)$.
With the same notations as above,
we hope to define the Okounkov body of $(X, L)$ $\Delta_v(X, L)$
to be the Okounkov body of the algebraic reduction $(A,\frac{1}{k}D)$ which is defined in \cite{LM09} (Definition 4.3).
The difficulty is whether this definition depends on the choice of $A$ and the algebraic reduction.
If $X$ is Moishezon, it can be easily checked that it is the case.
}
\end{myrem}

However, we can still study some asymptotic behaviour of the canonical ring.

For the reader's convenience, we briefly recall the construction of the Okounkov body in the projective case.
Assume $\xi$ is the centre of $v$ over $A$.
(Its existence is deduced from the properness of $A$.)
Assume that $D$ is a line bundle over $A$.
For any $\sigma \in H^0(A, mD) \setminus \{0\}$,
we define naturally the valuation of $\sigma$ associated to $v$ as follows.
Let $e$ be a local trivialisation of $\cO(D)$ near $\xi$.
Then there exists a local holomorphic function $f$ such that $\sigma=f \cdot e$ near $\xi$ (over a Zariski open set).
Define 
$$v(\sigma):= v(f)$$
which can be easily shown to be independent of the choice of local trivialisation.

Denote $\Lambda_v=v(\cM(A)^*)$ which is a lattice in $V_v := v(\cM(A)^*) \otimes_\Z \R$.
Define in $V_v$,
the Okounkov body $\Delta_v(A, D)$ associated to $D$ as the closure (with respect to the euclidean topology) of all $\frac{1}{m} v(\sigma)$ for $\sigma \in H^0(A, mD) \setminus \{0\}$.
It can be proven to be equal to the closure of all $v(E)$ where $E$ is an effective $\Q-$divisor $\Q-$linearly equivalent to $D$ with respect to the euclidean topology on $V_v$.
Here for an irreducible divisor $E$,
we define $v(E)$ as the valuation $v$ of any local defining function of the divisor $E$.
We can extend by linearity to define the valuation $v$ of any $\Q-$divisor.
Thus we can extend the definition of Okounkov body to the equivalent class of $\Q-$divisors.




\begin{myex}
{\em 
Let $T$ be a generic torus such that $\cM(T)=\C$ the constant functions.
Let $X$ be the blow-up of a point in $T \times \P^n$.
Then the composition $\pi$ of the blow-up and the projection onto $\P^n$ gives the algebraic reduction of $X$.
In particular, we have that
$$\cM(X) \cong \cM(T \times \P^n) \cong \cM(\P^n).$$
Consider $L:= \pi^* \cO(1) \otimes \cO(E)$ where $E$ is the exceptional divisor of the blow up.
Then the construction of Theorem 1 provides for any $k>0$, $\frac{1}{k} D$ is $\Q-$linear equivalent to $\cO(1)$ such that for any $m \geq 0$,
$$H^0(X, mL)=H^0(\P^n, \cO(m)).$$
In this case, one may define the Okounkov body of $(X, L)$ as
$$\Delta_v(X,L):=\Delta_v(\P^n, \cO(m)).$$
}
\end{myex}
As a direct application of Theorem 1, we have that
\begin{mycor}
Let $L$ be a line bundle over a compact normal irreducible complex space $X$.
Then we have that the limit
$$\lim_{k \in \N(L), k \to \infty} \frac{h^0(X,kL)}{k^{\kappa(L)}}$$
exists where $\kappa(L)$ is the Kodaira-Iitaka dimension of $L$.
\end{mycor}
\begin{proof}
It is a direct application of Theorem 1 and the corresponding result in the projective case.
We briefly sketch the proof of the projective case for the reader's convenience.

Take the same notations as in Theorem 1.

Recall that the rational rank of $v$ is defined as the rank of $\Lambda_v$, which is the maximal number of $\Z-$linear independent elements in $\Lambda_v$.
It can be shown that the rational rank of $v$ is less than the dimension of the algebraic reduction $A$, which is also equal to the transcendental degree of $\cM(X)$ over $\C$.
Fix $v$ a valuation with the maximal rational rank, which is always possible.

Let $\mu_v$ be the Lebesgue measure on $V_v$ normalised by the lattice $\Lambda_v$.
By Theorem 1, there exists $k_0$ sufficient divisible such that $k_0 L$ is effective, and there is an effective line bundle $D$ over $A$
such that
$$H^0(A, mD) \cong H^0(X, mk_0L) (\forall m \geq 0).$$
In particular,
$$\kappa(D)=\kappa(k_0L)=\kappa(L).$$
By the theory of Okounkov body, we have that for all valuation $v$ with maximal rational rank, 
$$\lim_{k k_0 \in \N(L), k \to \infty} \frac{h^0(X,kk_0 L)}{(kk_0)^{\kappa( L)}}=\lim_{k k_0 \in \N(L), k \to \infty} \frac{h^0(A,k D)}{(kk_0)^{\kappa( D)}}=\mu_v(\Delta_v(A,D)) k_0^{-\kappa(L)}.$$

Since $\N(L)$ is a semi-group, there exists $d$ large enough such that 
$$\N(L) \cap [k_1 d ,\infty [ =k_1 \N \cap [k_1 d, \infty[$$
for some $k_1 >0$.
Without loss of generality, we can assume that $k_0$ is a multiple of $k_1 d$. 
In particular $k_0 L, k_0  L+ k_1 L, \cdots, k_0 L+(k_0-k_1)L$ are all effective and we have inclusions for any $k\geq 1 $ and any $0 \leq i < k_0/k_1$ that
$$\cO((k-1 )k_0  L) \subset \cO(kk_0 L+k_1 i L) \subset \cO((k+2) k_0 L).$$
Thus we have for $0 \leq i \leq k_0/k_1-1$, 
$$\lim_{ k \to \infty} \frac{h^0(X,kk_0L+k_1 i L)}{(kk_0+k_1 i)^{\kappa( L)}}=\mu_v(\Delta_v(A,D))k_0^{-\kappa(L)}.$$
It implies the conclusion that
$$\lim_{k \geq d, k \to \infty} \frac{h^0(X,k_1kL)}{(k_1k)^{\kappa( L)}}=\lim_{ k \to \infty} \frac{h^0(X,kk_0 L+k_1 i L)}{(kk_0+k_1 i)^{\kappa( L)}}=\mu_v(\Delta_v(A,D))k_0^{-\kappa(L)}$$
since the right-hand side is independent of $i$.
\end{proof}
The existence of such limit is also previously studied in \cite{Iit71}, \cite{DEL00}.
One can also show the differentability of volume function on a Moishezon manifold. 
To show it, we need the following observation on the definition of movable intersection product in \cite{Bou02}, \cite{BEGZ},  \cite{BDPP} on a compact K\"ahler manifold.
\begin{myrem}
{\em 
Let $(Y, \omega)$ be a compact K\"ahler manifold.
Let $\pi:\tilde{Y} \to Y$ be a K\"ahler modification.
Let $\alpha_j$ be big classes on $Y$ such that $\pi^* \alpha_j$ are still big classes on $\tilde{Y}$.
(For example, this is the case when $\alpha_j$ are the first Chern classes of big line bundles over $Y$.)
By the construction of the movable positive product over a compact K\"ahler manifold,
we have 
$$\pi_* \langle \pi^* \alpha_1, \cdots, \pi^* \alpha_k \rangle=\langle  \alpha_1, \cdots,  \alpha_k \rangle.$$
In particular, let $L$ be a big line bundle over a Moishezon manifold $X$.
We can define the movable positive product of $c_1(L)$ as follows.
Let $\pi:\tilde{X} \to X$ be a modification of $X$ such that $\tilde{X}$ is a projective manifold.
Define for any $p >0$,
$$\langle  c_1(L)^p \rangle:=\pi_* \langle \pi^* c_1(L)^p \rangle$$
defined in $H^{p,p}_{BC}(X, \C)$.
By the filtration property of the modification, we can easily check that the product is independent of the choice of modification.
In other words, we have the same product for the push forward from any modification fo $X$ such that $\tilde{X}$ is a projective manifold.
}
\end{myrem}
\begin{myrem}
{\em 
Let $L$ be a big line bundle over a compact Moishezon manifold $X$ of dimension $n$.
Then for any $\xi \in \mathrm{NS}(X) \otimes_\Z \Q$, we have
$$\lim_{t \in \Q, t \to 0+} \frac{\Vol(L+t \xi)-\Vol(L)}{t}=n\langle c_1(L)^{n-1} \rangle \cdot c_1(\xi)$$
when the movable positive product is defined as in the previous remark.
The proof uses the birational invariance of the volume and reduces the case to a smooth projective bimeromorphic model.
The projective case is proven in \cite{BFJ}.

One can also easily get a weak version of Theorem 1.6 and Remark 3.5 of \cite{LX19}. Let $L_1, \cdots, L_n$ be big line bundles over a compact Moishezon manifold $X$ of dimension $n$, then we
have
$$\Vol(L_1+L_2)^{1/n}\geq \Vol(L_1)^{1/n}+\Vol(L_2)^{1/n},\langle c_1( L_1), \cdots,c_1( L_n )\rangle \geq \Vol(L_1)^{1/n} \cdots \Vol(L_n)^{1/n}.$$
If the equality is obtained in each above inequalities,  $\langle c_1(L_i) \rangle (\forall i)$ are proportional.
}
\end{myrem}

\textbf{Acknowledgement} I thank Jean-Pierre Demailly, my PhD supervisor, for his guidance, patience and generosity. 
I would like to thank my post-doc mentor Mihai P\u{a}un for many supports.
I would like to thank Sébastien Boucksom for some very useful suggestions on this objective.
In particular, I warmly thank Professeur Campana for providing the essential result in this note, allowing me to use it and reading and improving the manuscript.
I would also like to express my gratitude to colleagues of Institut Fourier for all the interesting discussions we had. This work is supported by the European Research Council grant ALKAGE number 670846 managed by J.-P. Demailly and DFG Projekt Singuläre hermitianische Metriken für Vektorbündel und Erweiterung kanonischer Abschnitte managed by Mihai P\u{a}un..
  

\begin{thebibliography}{9}
\bibitem[AS71]{AS71}
Andreotti, A. and W. Stoll.
Analytic and Algebraic Dependence of Meromorphic Functions, Lecture Note in Math. 234, Berlin-Heidelberg-New York : Springer 1971.
\bibitem[BDPP13]{BDPP}
S\'ebastien Boucksom, Jean-Pierre Demailly, Mihai P$\check{a}$un and Thomas Peternell, 
{\em The pseudoeffective cone of a
compact Kähler manifold and varieties of negative Kodaira dimension, } arXiv:math/0405285,
J. Algebraic Geom. 22, no. 2, 201-248 (2013).
\bibitem[BEGZ10]{BEGZ}
S\'ebastien Boucksom, Philippe Eyssidieux, Vincent Guedj, Ahmed Zeriahi, {\em Monge-Amp\`ere equations in big
cohomology classes,} Acta Math. 205 (2010), no. 2, 199–262.
\bibitem[BFJ09]{BFJ}
S\'ebastien Boucksom, Charles Favre and 
Mattias  Jonsson,
{\em 
Differentiability of volumes of divisors and a problem of Teissier,}
J. Algebraic Geom. 18, no. 2, 279-308 (2009) 
\bibitem[Bou02]{Bou02}
S\'ebastien Boucksom,
{\em Divisorial Zariski decompositions on compact complex manifolds.}
Ann. Sci. ENS (4) 37, no. 1, 45-76 (2004).
\bibitem[Cam81]{Cam81}
Frédéric Campana,
{\em 
Réduction algébrique d'un morphisme faiblement Kählérien propre et applications.}
Mathematische Annalen volume 256, pages 157–189(1981).
\bibitem[Cam82]{Cam82}
Frédéric Campana,
{\em Sur les diviseurs non-polaires d'un espace analytique compact,}
J Reine Angew. Math. 332 (1982), 126-133. 
\bibitem[Cam04]{Cam04}
Frédéric Campana,
{\em Special Varieties and classification Theory,}
Annales de l'Institut Fourier 54, 3 (2004), 499-665.
\bibitem[DEL00]{DEL00}
Jean-Pierre Demailly, Lawrence Ein, and Robert Lazarsfeld,
{\em 
 A subadditivity property of multiplier ideals,} math.AG/0002035, Michigan Math. J. (special volume in honour of W. Fulton), 48 (2000), 137-156.
\bibitem[FF79]{FF79}
G. Fischer, O. Forster, Ein Endlichkeitssatz für Hyperflächen auf kompakten komplexen Räumen, J. reine
angew. Math. 306 (1979), 88—93.
\bibitem[Hir64]{Hir64}
H. Hironaka, Resolution of singularities of an algebraic variety over a field
of characteristic zero, Ann. of Math. 79 (1964), 109–326.
\bibitem[Hir75]{Hir75}
H. Hironaka,
 Flattening theorem in complex-analytic geometry, Am. J. Math. 97
(1975), 503–547.
\bibitem[Iit71]{Iit71}
Shigeru Iitaka,
{\em 
On D-dimensions of algebraic varieties.} J. Math. Soc. Japan
23 (1971), 356–373.
\bibitem[KK12]{KK12}
Kiumars Kaveh and  A. G. Khovanskii,
{\em
Newton-Okounkov bodies, semigroups of integral points, graded algebras and intersection theory.}
Annals of Mathematics176(2012), 925–978.
\bibitem[LM09]{LM09}
Robert Lazarsfeld and
Mircea Mustaţǎ,
{\em Convex bodies associated to linear series.}
 Ann. Sci. Éc. Norm. Supér. (4) 42 (2009), 783–835.
\bibitem[LX19]{LX19}
B. Lehmannand and Jian Xiao,
{\em  Positivity functions for curves on algebraic varieties.}  Algebra Number Theory 13(6): 1243-1279 (2019). DOI: 10.2140/ant.2019.13.1243 
\bibitem[Oko96]{Oko96}
Andre\"i Okounkov,
{\em
Brunn-Minkowski inequality for multiplicities.} Invent.
math. 125 (1996), 405–411.
\bibitem[Oko00]{Oko00}
Andre\"i Okounkov,
{\em Why would multiplicities be log-concave?} in The orbit
method in geometry and physics (Marseille, 2000), 329–347. Progr. Math.
213, Birkhäuser Boston, Boston, MA, 2003.
\bibitem[Rem56]{Rem56}
Reinhold Remmert, 
{\em Meromorphe Funktionen in Kompakten komplexen R\"aume,} Math Ann , 132 (1956), 277-288. 
\bibitem[Thi54]{Thi54}
Walter Thimm,
{\em 
Meromorphe Abbildungen yon Riemannschen Bereichen.} Math. Zeit. 60 (1954), 435-457.
\bibitem[Ueno75]{Ueno}
Kenji Ueno,
{\em  Classification theory of algebraic varieties and compact complex spaces.} In: Lecture Notes in Mathematics, Vol. 493. Berlin, Heidelberg, New York: Springer 1975.
\end{thebibliography}
\end{document}